\documentclass[reqno,centertags,11pt]{amsart}

\usepackage{amssymb,amsmath,amsfonts,amssymb}
\usepackage{hyperref}
\usepackage{enumerate}

-\marginparwidth \oddsidemargin -9mm \evensidemargin \oddsidemargin
\usepackage{latexsym}
\advance\hoffset by 5mm

\def\@abssec#1{\vspace{.05in}\footnotesize \parindent .2in
{\bf #1. }\ignorespaces}

\newtheorem{theorem}{Theorem}[section]

\newtheorem{lemma}[theorem]{Lemma}

\newtheorem{remark}[theorem]{Remark}

\allowdisplaybreaks \numberwithin{equation}{section}

\begin{document}

\title[locally self-similar solutions for Euler equations]
{On the locally self-similar singular solutions for the incompressible Euler equations}

\author{Liutang Xue}
\address{School of Mathematical Sciences, Beijing Normal University and Laboratory of Mathematics and Complex Systems, Ministry of Education, Beijing 100875, P.R. China}
\address{Universit\'e Paris-Est, CERMICS, Ecole des Ponts ParisTech, Cit\'e Descartes, Champs-sur-Marne, 77455 Marne-la-Vall\'ee Cedex 2, France}
\email{xuelt@bnu.edu.cn}
\date{}
\maketitle

\begin{abstract}
  In this paper we consider the locally backward self-similar solutions for the Euler system, and focus on the case that the possible nontrivial velocity profiles have non-decaying asymptotics.
We derive the meaningful representation formula of the pressure profile in terms of velocity profiles in this case, and by using it and the local energy inequality of profiles,
we prove some nonexistence results and show the energy behavior concerning the possible velocity profiles.
\end{abstract}

\section{Introduction}

Perfect incompressible fluids are governed by the well-known Euler system
\begin{equation}\label{Euler}
\begin{cases}
  \partial_t v + v \cdot\nabla v + \nabla p =0, \qquad  \\
  \nabla_x \cdot v =0,\\
  v|_{t=0}=v_0(x),
\end{cases}
\end{equation}
where $(x,t)\in \mathbb{R}^N \times \mathbb{R}^+$, $N=2,3,\cdots$ is the spatial dimension, $v=(v_1,v_2,\cdots,v_N)$ is the velocity vector field of $\mathbb{R}^N$ and $p$ is the scalar-valued pressure field.
Assume $v_0\in H^s(\mathbb{R}^N)$, $s>\frac{N}{2}+1$,
it has been known for decades (e.g. \cite{Kato72}) that there is a unique local-in-time smooth solution $v\in C([0,T[; H^s)$ and the pressure can be expressed up to a constant by
$p=-\mathrm{div} \,\mathrm{div}\Delta^{-1}(v\otimes v)$, that is,
\begin{equation}\label{eq:pexp}
  p(x,t)= -\frac{1}{N} |v(x,t)|^2 + p.v. \int_{\mathbb{R}^N} K_{ij}(x-y) v_i(y,t)v_j(y,t)\,\mathrm{d}y,
\end{equation}
where $K_{ij}(y)= \frac{1}{N|\mathbb{S}^{N-1}|}\frac{Ny_i y_j- |y|^2 \delta_{ij}}{|y|^{N+2}}$ ($i,j=1,2,\cdots,N$)
is the Calder\'on-Zygmund kernel and the Einstein convention on repeated indices is used here and thereafter.
However, for $N\geq 3$, whether such smooth solutions are globally well-posed or they have finite-time blowup remains a challenging open problem.

In this paper we address the problem of the existence or not of the locally backward self-similar solutions for Euler system.
More precisely, we consider solutions that develop a
finite-time self-similar singularity on a spacetime domain $]0,T[\times B_\rho(x_0)$ of the form
\begin{equation}\label{eq:vpSelf}
  v(x,t)= \frac{1}{(T-t)^{\frac{\alpha}{\alpha+1}}} u\bigg(\frac{x-x_0}{(T-t)^{\frac{1}{\alpha+1}}}\bigg),
\end{equation}
and
\begin{equation}\label{eq:vpSelf2}
  p(x,t)= \frac{1}{(T-t)^{\frac{2\alpha}{\alpha+1}}} q \bigg(\frac{x-x_0}{(T-t)^{\frac{1}{\alpha+1}}}\bigg)+ d(t),
\end{equation}
where $(u,q)$ are stationary profiles, $T>0$, $\alpha>-1$, $x_0\in\mathbb{R}^N$, $\rho>0$, and the solutions $v$, $p$ remain regular outside the ball $B_\rho(x_0)$. 
If $\rho=\infty$, i.e. $B_\rho(x_0)=\mathbb{R}^N$, this corresponds to the ``globally" self-similar solutions and $d(t)\equiv0$;
while if $\rho<\infty$, these are the ``locally" self-similar solutions, and $d(t)$ is a function depending only on $t$.
For the locally self-similar solutions, from \eqref{eq:pexp} and \eqref{eq:vpSelf}, it seems not obvious to get the expression \eqref{eq:vpSelf2},
but which can indeed be justified by Lemma \ref{lem:pq} below. In terms of $(u,q)$, we formally have
\begin{equation}\label{EulerSelf}
\begin{cases}
  \frac{\alpha}{\alpha+1} u +\frac{1}{\alpha+1}y\cdot\nabla u + u\cdot\nabla u + \nabla q=0, \\
  \mathrm{div}\,u=0,
\end{cases}
\end{equation}
where $y\in\mathbb{R}^N$ and $ q $ up to a harmonic polynomial is given by $\Delta q=-\mathrm{div}\mathrm{div}(u\otimes u)$, more precisely, $q$ up to a constant is given by \eqref{eq:qy} below according to the value of $\alpha$
and the asymptotic assumptions of $u$.


The self-similar ansatz \eqref{eq:vpSelf}-\eqref{eq:vpSelf2} for the Euler system \eqref{Euler} is widely used in the numerical simulations,
and through studying the vortex filament models or high-symmetric flows, much work suggests that such backward self-similar blowup may happen at a finite time
(see e.g. \cite{BorP,Kerr,Kimu,Pelz}, and very recent work, \cite{HouL,HouL2}).

The self-similar singular solutions of Euler equations are also studied from the analytical viewpoint.
X. He in \cite{He07} constructed non-trivial solutions to the 3D Euler equations \eqref{EulerSelf} with $\alpha=1$ on the exterior domain
$\mathbb{R}^3\setminus B_1(0)$, and the asymptotic decay of such solutions are $|u(y)| \lesssim \frac{1}{|y|}$
and $|\nabla u(y)|\lesssim \frac{1}{|y|^2}$.
Besides that, there are some noticeable nonexistence results on such self-similar solutions in the literature.
D. Chae in \cite{Chae1} considered the globally self-similar solutions to the 3D Euler system and proved that if
$u\in C^1(\mathbb{R}^3)$ and $\omega=\nabla \times u$ belongs to $\cap_{0<r<r_0} L^r(\mathbb{R}^3)$ with some $r_0>0$,
then $\omega\equiv 0$ for all $\alpha>-1$. 
R. Takada \cite{Tak} treated the strong solutions of the self-similar Euler equations \eqref{EulerSelf}
and show $u\equiv 0$ under the condition $u\in C^1_{\mathrm{loc}}\cap X^{2,\infty}\cap L^p$ with $p\in [\frac{3N}{N-1},\frac{4N}{N-2}]$
and $X^{2,\infty}= \{f\in L^2_{\mathrm{loc}}: \sup_R \int_{R\leq |y|\leq 2R} |f(y)|^2\mathrm{d}y <\infty\}$.
See also \cite{He00,Scho} for similar nonexistence results.
For the locally self-similar solutions \eqref{eq:vpSelf}-\eqref{eq:vpSelf2} with $\rho>0$,
D. Chae and R. Shvydkoy \cite{ChaS} proved that if $u\in C^1_{\mathrm{loc}}\cap L^r$
with $r\in [3,\infty]$, then $u\equiv 0$ for all $-1<\alpha <\frac{N}{r}$ and $\alpha>\frac{N}{2}$.
They also improved the result of \cite{Chae1} to get $u\equiv const$ for all $\alpha>-1$
under the assumptions $u\in C^1_{\textrm{loc}}(\mathbb{R}^N)$, $\omega=\nabla\times u\in L^p(\mathbb{R}^N)$ for some $p\in ]0,\frac{N}{1+\alpha}[$, and $|\nabla u(y)|=0(1)$ as $|y|\rightarrow \infty$.
Recently, A. Bronzi and R. Shvydkoy in \cite{BroS} rigorously justified the formula of pressure \eqref{eq:vpSelf2} for the locally self-similar solutions at the case $\alpha>0$ and $\rho>0$,
and under the assumptions $u\in C^3_{\mathrm{loc}}(\mathbb{R}^N)$ and
$$\textrm{for some  }p\geq 3,\gamma<p-2,\quad \int_{|y|\sim L}|u(y)|^p\,\mathrm{d}y\lesssim L^\gamma,\quad \forall L\gg1,\;$$
they proved at the case $0<\alpha<N/2$ either $u\equiv 0$ or the corresponding velocity profiles behave as \eqref{eq:conc}.


In this article we deal with the locally backward self-similar solutions \eqref{eq:vpSelf}-\eqref{eq:vpSelf2}
of Euler equations \eqref{Euler} to show some nonexistence results, and we specially are concerned with the situation that the velocity profiles $u$ have non-decaying asymptotics,
e.g. for some $\delta\in ]0,1[$,
\begin{equation}\label{eq:asum}
  1\lesssim |u(y)| \lesssim |y|^\delta,\quad \forall |y|\gg1.
\end{equation}
This case is not much addressed in the literature (except for the implicitly related nonexistence results based on vorticity profile, e.g. \cite{ChaS}), but it is motivated by the numerical simulations (e.g. \cite{HouL,HouL2})
and especially by several works on the 1D models of Euler equations:
the 1D Burgers equation, the 1D CCF model (cf. \cite{CCF}), the 1D CKY model (cf. \cite{CKY}) and so on. The blowup issue of all these 1D equations is clear:
the Burgers equation develops shock singularity at finite time,
while it is proved in \cite{CCF} and \cite{CKY} respectively that the CCF equation and the CKY equations form finite-time singularity for some smooth data.
The further study (\cite{EggF}, \cite{dlHF}, \cite{HouLiu}, respectively)
shows that the finite-time singularities of all these equations are of locally self-similar type with some index $\alpha\in ]-1,0[$ and the corresponding velocity profiles have growing spacial asymptotics.
Due to the formal analogy of these 1D models with the real Euler equations, it deserves much to consider such a scenario for the Euler equations \eqref{Euler}.

In order to do so, we have to derive a new and meaningful representation formula of the pressure profile in terms of velocity profiles, since the usual one
\begin{equation}
  q(y)=-\frac{1}{N}|u(y)|^2 + p.v. \int_{\mathbb{R}^N} K_{ij}(x-y) u_i(y)u_j(y)\,\mathrm{d}y,
\end{equation}
works for the velocity profiles with suitable decaying asymptotics, e.g. $u\in L^p(\mathbb{R}^N)$ for some $p\in]2,\infty[$, but it does not make sense for the profiles satisfying \eqref{eq:asum}.
We justify the needed formula in Lemma \ref{lem:pq}, which is stated as
\begin{equation}\label{eq:qy0}
  q(y)= -\frac{1}{N} |u(y)|^2 + p.v.\int_{\mathbb{R}^N} K_{ij}(y-z) u_i(z) u_j(z)\,\mathrm{d}z + \bar{q}(y) + A\cdot y,
\end{equation}
where $A\in \mathbb{R}^N$ is some fixed constant vector and
\begin{equation}
\bar{q}(y)=
  \begin{cases}
    -\int_{|z|\geq M}K_{ij}(z) u_i(z) u_j(z)\,\mathrm{d}z,\quad & \textrm{if}\;\;1\lesssim |u(z)| \lesssim |z|^\delta,\delta\in [0,\frac{1}{2}[,\,  \\
    -\int_{|z|\geq M} \big(K_{ij}(z)+ y\cdot\nabla K_{ij}(z)\big)u_i(z) u_j(z)\,\mathrm{d}z ,\quad & \textrm{if}\;\; |z|^{\frac{1}{2}}\lesssim |u(z)|\lesssim |z|^\delta,\delta\in[\frac{1}{2},1[,
  \end{cases}
\end{equation}
and $M>0$ is a large number so that \eqref{eq:asum} holds for all $|y|\geq M$.
The formula \eqref{eq:qy0} is also expressed as the decomposition \eqref{eq:qy-dec} (and its variant), which can be used to show that $q(y)$ belongs to $C^2_{\textrm{loc}}(\mathbb{R}^N)$
as long as $u\in C^3_{\textrm{loc}}(\mathbb{R}^N)$ (see Lemma \ref{lem:pq}).

Our main results read as follows.
\begin{theorem}\label{thm:SS}
Suppose that $\alpha>-1$, $u\in C^3_{\textrm{loc}}(\mathbb{R}^N)$ satisfies that for some $\delta\in ]0,1[$,
\begin{equation}\label{eq:asum2}
  |u(y)| \lesssim |y|^\delta,\quad \forall |y|\gg1.
\end{equation}
and up to a constant $q$ is defined from $u$ by \eqref{eq:qy} below. We have the following statements.
\begin{enumerate}
\item
If additionally there is a small number $0<\epsilon_0\ll \delta$ so that
\begin{equation}\label{eq:asum3}
  |u(y)| \gtrsim |y|^{\epsilon_0}, \quad\forall |y|\gg 1,
\end{equation}
then the possible scope of $\alpha$ to admit nontrivial self-similar velocity profiles is $-\delta\leq \alpha\leq -\epsilon_0$,
and for each $\alpha$, the corresponding profiles satisfy that
\begin{equation}\label{eq:conc}
  \int_{|y|\leq L}|u(y)|^2 \mathrm{d}y \sim L^{N-2\alpha},\quad \forall L\gg1.
\end{equation}
\item
If $\delta<\frac{1}{2}$, $\alpha>-\frac{1}{2}$, and additionally
\begin{equation}\label{eq:asum4}
  |u(y)|\gtrsim 1, \quad\forall |y|\gg 1,
\end{equation}
then the possible range of $\alpha$ to admit nontrivial self-similar profiles is $-\delta\leq\alpha\leq 0$,
and the velocity profiles corresponding to each $\alpha$ satisfy \eqref{eq:conc}.
\end{enumerate}
\end{theorem}

In the proof of Theorem \ref{thm:SS}, we start with the following local energy inequality of the profiles $(u,q)$ (cf. \cite{ChaS})
for $0<l_1< l_2$ and the standard test function $\phi$,
\begin{equation}\label{eq:locEne}
\begin{split}
  \Big|\frac{1}{l_2^{N-2\alpha}}\int_{|y|\leq l_2} |u(y)|^2 \phi\Big(\frac{y}{l_2}\Big)\mathrm{d}y -
  \frac{1}{l_1^{N-2\alpha}}\int_{|y|\leq l_1}|u(y)|^2 \phi\Big(\frac{y}{l_1}\Big)\mathrm{d}y\Big|
  \leq\, C \int_{l_1/2 \leq |y|\leq l_2} \frac{|u|^3 + |q| |u|}{|y|^{N-2\alpha+1}} \mathrm{d}y,
\end{split}
\end{equation}
and by applying the bootstrapping method and a careful analysis according to the values of $(\alpha,\delta)$, we show the main results, which is placed in Section \ref{sec:thm:SS}.
In this process, the treating of the term containing pressure profile is technical and is used repeatedly, and we write it as Lemma \ref{lem:pres} for convenience, which is stated and proved in Section \ref{sec:lem}.

\begin{remark}
  In Theorem \ref{thm:SS}-(1), the assumption \eqref{eq:asum3} can be replaced by that for every $0<\epsilon_0\ll \delta$,
\begin{equation*}
  \int_{|y|\leq L} |u(y)|^2\,\mathrm{d}y \gtrsim L^{N+2\epsilon_0},\quad \forall L\gg1,
\end{equation*}
and the same conclusion can be obtained. In Theorem \ref{thm:SS}-(2), the assumption \eqref{eq:asum4} can be weakened to be that for every $0<\epsilon_0\ll 1$,
\begin{equation*}\label{eq:assum5}
  \int_{|y|\leq L} |u(y)|^2\,\mathrm{d}y \gtrsim L^{N-2+2\epsilon_0},\quad \forall L\gg1,
\end{equation*}
then the possible range of $\alpha$ admitting nontrivial profiles is $-\delta\leq \alpha \leq 1-\epsilon_0$, and the profiles corresponding to each $\alpha$ satisfy \eqref{eq:conc}.
\end{remark}

\begin{remark}
From \eqref{eq:conc} for every $-1<\alpha<\frac{N}{2}$ (cf. \cite{BroS} for the case $0<\alpha<\frac{N}{2}$ and cf. Theorem \ref{thm:SS} for the case $-1< \alpha \leq 0$),
we can expect the ``typical" possible velocity profiles are that
\begin{equation}\label{eq:TypVel}
  |u(y)| \sim |y|^{-\alpha} + l.o.t.,\quad \forall |y|\gg1,\,
\end{equation}
where $l.o.t.$ is the abbreviation of the lower order terms.
By scaling, we can also expect the typical vorticity profiles are
\begin{equation}\label{eq:typVor}
  |\nabla\times u(y)| \sim |y|^{-\alpha-1}+l.o.t.,\quad \forall |y|\gg1, \,
\end{equation}
which is compatible with the nonexistence result based on the vorticity profile of \cite{ChaS}.
Furthermore, in the considered blowup scenario and using \eqref{eq:typVor}, we have that for all
$(t,x)\in ]0,T[\times \big(B_\rho(x_0)\setminus\{x_0\}\big)$,
\begin{equation}
  \nabla\times v(x,t)= \frac{1}{T-t} \nabla\times u\bigg(\frac{x-x_0}{(T-t)^{\frac{1}{1+\alpha}}}\bigg)
  \sim  \frac{1}{|x-x_0|^{1+\alpha}}.
\end{equation}
Such typical self-similar blowup case is consistent with the Beale-Kato-Majda criterion (cf. \cite{BKM}) since for all $\alpha>-1$,
\begin{equation*}
  \int_0^T \|\nabla\times v\|_{L^\infty}\,\mathrm{d}t \sim T
  \sup_{0<|x-x_0|\leq \rho}|x-x_0|^{-(1+\alpha)}=\infty.
\end{equation*}
On the other hand, the bound $\int_0^T\|\nabla\times v\|_{L^r}\,\mathrm{d}t<\infty$ with some $1\leq r<\infty$
is not sufficient to get rid of such typical blowup scenario \eqref{eq:TypVel}-\eqref{eq:typVor} for all $\alpha>-1$; indeed,
this typical blowup scenario still may happen at the range $-1<\alpha<-1+N/r$.
\end{remark}



Throughout this paper, $C$ stands for a constant which may be different from line to line,
$X\lesssim Y$ means that there is a harmless constant $C$ such that $X\leq C Y$, and $X\sim Y$ means that $X\lesssim Y$ and $Y\lesssim X$ simultaneously.
Denote $B_r(x):=\{y\in \mathbb{R}^N: |y-x|\leq r\}$ the ball of $\mathbb{R}^N$
and $B^c_r(x):=\mathbb{R}^N\setminus B_r(x)$ its complement set. For a number $a\in \mathbb{R}$, denote $[a]$ by the interger part of $a$.

\section{Auxiliary lemmas of the pressure profile}\label{sec:lem}

We collect two auxiliary lemmas in this section: one is about the justification of the representation formula of pressure, and the other is about the estimating the term containing pressure profile, which is useful in the main proof.

\begin{lemma}\label{lem:pq}
  Suppose $v$ is a locally self-similar solution \eqref{eq:vpSelf} to the Euler equations and $\alpha>-1$.
Assume that $u\in C^3_{\mathrm{loc}}(\mathbb{R}^N)$ satisfies that for some $\delta\in [0,1[$,
\begin{equation}\label{eq:ucond}
  |u(y)|\lesssim |y|^\delta,\quad \forall |y|\geq M,
\end{equation}
with $M>0$ a large number,
then the corresponding pressure on the ball $B_{\rho}(x_0)$ for all $t$ near $T$ is expressed as
\begin{equation}\label{eq:keyeq2}
  p(x,t)=\frac{1}{(T-t)^{\frac{2\alpha}{\alpha+1}}} q\bigg(\frac{x-x_0}{(T-t)^{\frac{1}{\alpha+1}}}\bigg) +d(t),
\end{equation}
where $d(t)$ is a function depending only on $t$ satisfying \eqref{eq:dt}, and $q(y)$ is a $C^2_{\textrm{loc}}$-smooth scalar function defined by
\begin{equation}\label{eq:qy}
  q(y)= -\frac{1}{N} |u(y)|^2 + p.v.\int_{\mathbb{R}^N} K_{ij}(y-z) u_i(z) u_j(z)\,\mathrm{d}z + \bar{q}(y) + A\cdot y
\end{equation}
with $A\in\mathbb{R}^N$ some fixed constant vector (especially, $A$ equals 0 if $\alpha>-\frac{1}{2}$, $\delta<1/2$ or $\alpha>-\frac{1}{2}$, $u\in L^p(\mathbb{R}^N)$, $p\in [2,\infty[$) and $\bar{q}(y)$ given by
\begin{equation}
  \begin{cases}
    0,\quad & \textrm{if}\;\; u\in L^p(\mathbb{R}^N), p\in [2,\infty[, \\
    -\int_{|z|\geq M}K_{ij}(z) u_i(z) u_j(z)\,\mathrm{d}z,\quad & \textrm{if}\;\;1\lesssim |u(z)| \lesssim |z|^\delta,\delta\in [0,\frac{1}{2}[,\,\forall |z|\geq M,  \\
    -\int_{|z|\geq M} \big(K_{ij}(z)+ y\cdot\nabla K_{ij}(z)\big)u_i(z) u_j(z)\,\mathrm{d}z ,\quad & \textrm{if}\;\; |z|^{\frac{1}{2}}\lesssim |u(z)|\lesssim |z|^\delta,\delta\in[\frac{1}{2},1[,\forall |z|\geq M.
  \end{cases}
\end{equation}

\end{lemma}

\begin{proof}[Proof of Lemma \ref{lem:pq}]
We here mainly adapt the strategy in the proof of \cite[Lemma 2.1]{BroS} with suitable modification.

First define the quantity contained in \eqref{eq:qy} as
\begin{equation}\label{eq:Iy}
  I(y)= -\frac{1}{N} |u(y)|^2 + p.v.\int_{\mathbb{R}^N} K_{ij}(y-z) u_i(z) u_j(z)\,\mathrm{d}z + \bar{q}(y),\quad
\end{equation}
and we show that $I(y)$ is meaningful and is a tempered distribution. Let $\phi_0\in \mathcal{D}(\mathbb{R}^N)$ satisfying $0\leq \phi_0\leq 1$ be a cutoff function supported on
$B_1(0)$ such that $\phi_0\equiv 1$ on $B_{1/2}(0)$. For any $L\geq M$, set
$\phi_L(z)=\phi_0(z/L)$, then we have
\begin{equation*}
\begin{split}
  p.v.\int_{\mathbb{R}^N} K_{ij}(y-z) u_i(z) u_j(z) \,\mathrm{d}z
  =\, & p.v. \int_{\mathbb{R}^N} K_{ij}(y-z) \phi_{4L}(z) u_i(z) u_j(z)\,\mathrm{d}z\\
  & + \int_{\mathbb{R}^N} K_{ij}(y-z) (1-\phi_{4L}(z)) u_i(z) u_j(z)\,\mathrm{d}z \\
  :=\, &  I_{1}(y,L) + I_2(y,L).
\end{split}
\end{equation*}
Since $u\in C_{\mathrm{loc}}^3(\mathbb{R}^3)$, from the Besov embedding, we infer that $I_1(y,L)\in C^\beta$ for all $\beta<3$.
We next consider $I_2(y,L)+ \bar q(y)$ acting on the ball $B_L(0)$,
and if $u\in L^p(\mathbb{R}^N)$ for some $p\in [2,\infty[$, then
\begin{equation}\label{eq:I2yL}
\begin{split}
  I_2(y,L) \lesssim  \sum_{k=0}^\infty \int_{2^k L\leq |z|\leq 2^{k+1} L} \frac{1}{|z|^N} |u(z)|^2 \mathrm{d}z
  \lesssim  \sum_{k=0}^\infty (2^k L)^{-N+ N(1- 2/p)} \|u\|_{L^p}^{2/p} \lesssim L^{-2N/p},
\end{split}
\end{equation}
and if $1\lesssim |u(z)| \lesssim |z|^\delta$, $\delta\in ]0,\frac{1}{2}[$ for all $ |z|\geq M$, then
\begin{equation}\label{eq:I2key}
\begin{split}
  I_2(y,L)+ \bar q(y)\leq & \,\Big|\int_{|z|\geq 4L}\big(K_{ij}(y-z)-K_{ij}(z)\big) u_i(z) u_j(z)\mathrm{d}z\Big| +  C\int_{M\leq |z|\leq 4L} \frac{1}{|z|^N} |u(z)|^2\,\mathrm{d}z \\
  \lesssim & \int_{|z|\geq 2L} \frac{|y|}{|z|^{N+1}} |u(z)|^2\,\mathrm{d}z+ \int_{M\leq |z|\leq 4 L} |z|^{-N+2\delta} \mathrm{d}z\lesssim L^{2\delta},
\end{split}
\end{equation}
and if $|z|^{1/2}\lesssim |u(z)| \lesssim |z|^\delta$, $\delta\in [1/2,1[$ for all $ |z|\geq M$, then
\begin{equation}\label{eq:I2key2}
\begin{split}
  I_2(y,L)+ \bar q(y) \leq & \Big|\int_{|z|\geq 4L} \big(K_{ij}(y-z)-K_{ij}(z)-y\cdot\nabla K_{ij}(z)\big) u_i(z) u_j(z)\mathrm{d}z\Big| \\
  & \,+\, C\int_{M\leq |z|\leq 4 L} \Big( \frac{1}{|z|^N} + \frac{|y|}{|z|^{N+1}}\Big) |u(z)|^2 \mathrm{d}z \\
  \lesssim & \int_{|z|\geq 2L} \frac{|y|^2}{|z|^{N+2}} |u(z)|^2\,\mathrm{d}z+ \int_{|z|\sim L} \Big(\frac{1}{|z|^{N-2\delta}}+ \frac{|y|}{|z|^{N+1-2\delta}}\Big) \mathrm{d}z\lesssim L^{2\delta}.
\end{split}
\end{equation}
For $s=1,2$ and for all $y\in B_L(0)$, we also get that if $u\in L^p(\mathbb{R}^N)$, $p\in[2,\infty[$,
\begin{equation*}
\begin{split}
  \partial_y^s(I_2(y,L)) \lesssim \sum_{k=0}^\infty \int_{2^k L\leq |z|\leq 2^{k+1}L} \frac{1}{|z|^{N+s}} |u(z)|^2\,\mathrm{d}z \lesssim L^{-s-\frac{2N}{p}},
\end{split}
\end{equation*}
and if $1\lesssim |u(z)| \lesssim |z|^\delta$, $\delta\in ]0,1/2[$, $\forall |z|\geq M$,
\begin{equation*}
\begin{split}
  \partial^s_y \big(I_2(y,L)+ \bar{q}(y)\big)  &\leq \partial_y^s \Big(\int_{|z|\geq 4L} \int_0^1 y\cdot \nabla K_{ij}(\tau y-z) u_i(z)u_j(z)\,\mathrm{d}\tau\mathrm{d}z  \Big)
  + \int_{|z|\sim L} \frac{C }{|z|^{N+s}} |u(z)|^2\,\mathrm{d}z \\
  &\lesssim \int_{|z|\geq 2L} \frac{|y|}{|z|^{N+1+s}} |u(z)|^2\,\mathrm{d}z + \int_{|z|\sim L} \frac{1 }{|z|^{N+s}} |u(z)|^2\,\mathrm{d}z
  \lesssim L^{-s+2\delta},
\end{split}
\end{equation*}
and if $|z|^{1/2}\lesssim |u(z)| \lesssim |z|^\delta$, $\delta\in [1/2,1[$, $\forall |z|\geq M$,
\begin{equation*}
\begin{split}
  \partial^s\big(I_2(y,L)+ \bar q(y)\big) \leq &\partial^s_y\bigg(\int_{|z|\geq 4L} \int_0^1 \int_0^1 \Big(y\cdot \nabla^2 K_{ij}(\tau\theta y -z)\cdot y\Big)
  u_i(z) u_j(z)\tau\mathrm{d}\theta\mathrm{d}\tau\mathrm{d}z\bigg)  \\
  &  - \partial_y^s \bigg(\int_{M\leq |z|\leq 4L} y\cdot\nabla K_{ij}(z) u_i(z)u_j(z) \,\mathrm{d}z\bigg)+ C\int_{|z|\sim L} \frac{1}{|z|^{N+s}} |u(z)|^2 \mathrm{d}z   \\
  \lesssim & \int_{|z|\geq 2L} \frac{|y|^2}{|z|^{N+2+s}} |u(z)|^2\,\mathrm{d}z+ \int_{|z|\sim L} \frac{1}{|z|^{N+s}} |u(z)|^2 \mathrm{d}z
  \\ & \, +
  \begin{cases}
    \int_{M\leq |y|\leq 4L} \frac{1}{|z|^{N+1}} |u(z)|^2\,\mathrm{d}z,\quad & \textrm{if}\;\; s=1, \\
    0, \quad & \textrm{if}\;\; s=2,
  \end{cases} \\
  \lesssim &
    L^{-s+2\delta}.
\end{split}
\end{equation*}
Hence the scalar function $I(y)$ defined by \eqref{eq:Iy}
is $C^2$-smooth on $B_L(0)$.
Moreover, for all $y\in B_L(0)$, we have
\begin{equation*}
\begin{split}
  \Delta I & \,= \Delta \Big(-\frac{1}{N} |u|^2\phi_{4L}+I_1\Big) + \Delta \Big(I_2 + \bar{q}(y)\Big)\\
  & \,= -\mathrm{div}\, \mathrm{div}\big(u\sqrt{\phi_{4L}}\otimes u\sqrt{\phi_{4L}}\big)
  = -\mathrm{div}\, \mathrm{div}\big(u\otimes u\big),
\end{split}
\end{equation*}
where in the second line $\Delta( I_2+ \bar{q}(y))=0$ due to that the term $K_{ij}(y-z)-K_{ij}(z)-y\cdot\nabla K_{ij}(z)$ is harmonic in the $y$-variable for all
$y\in B_L(0)$ and $z\in B_{2L}^c(0)$.
Besides, it is not hard to show that $I$ is a tempered distribution, which can be seen from the following computation that for $L\gg1$ and some $r>2$,
\begin{equation*}
  \int_{|y|\leq L} |I_1(y,L)|^{\frac{r}{2}}\,\mathrm{d}y \lesssim
  \int_{|z|\leq 4L} |u(z)|^r\,\mathrm{d}z \lesssim L^{N+r\delta},
\end{equation*}
and by \eqref{eq:I2yL}-\eqref{eq:I2key2},
\begin{equation*}
\begin{split}
  \int_{|y|\leq L} |I_2(y,L)+\bar q(y)|^{\frac{r}{2}}\,\mathrm{d}y   \lesssim L^{N+ r\delta}.
\end{split}
\end{equation*}

Next we intend to find a distributional pressure profile solving the first equation of \eqref{EulerSelf}, i.e.,
\begin{equation}\label{eq:keyeq}
  \frac{\alpha}{\alpha+1} u +\frac{1}{\alpha+1}y\cdot\nabla u + u\cdot\nabla u + \nabla q=0,
\end{equation}
Applying the ansatz \eqref{eq:vpSelf} to the Euler equations \eqref{Euler}, and by setting
\begin{equation}
  y:=\frac{x-x_0}{(T-t)^{\frac{1}{1+\alpha}}},\quad p(x,t):= \bar p(y,t),
\end{equation}
we obtain that for all $|y|\leq \rho(T-t)^{-\frac{1}{1+\alpha}}$,
\begin{equation}\label{eq:ueq2}
  \frac{\alpha}{1+\alpha} u(y) + \frac{1}{1+\alpha} y\cdot\nabla_y u(y) + u\cdot\nabla_y u(y)
  + \nabla_y \Big((T-t)^{\frac{2\alpha}{1+\alpha}}\bar p(y,t)\Big)=0.
\end{equation}
For any fixed $t<T$, denoting $f(y,t)= (T-t)^{\frac{2\alpha}{1+\alpha}} \bar p(y,t)$, then the vector-valued function $\nabla_y f(y,t)=:g(y)$ depends only on $y$
on the domain $D(t):=\{y: |y|\leq \rho (T-t)^{-\frac{1}{1+\alpha}}\}$. Thus from the fundamental theorem of calculus,
we deduce that for all $y\in D(t)$,
\begin{equation}\label{eq:qexp}
\begin{split}
  & (T-t)^{\frac{2\alpha}{1+\alpha}}\bar p(y,t)-(T-t)^{\frac{2\alpha}{1+\alpha}}\bar p(0,t) \\
  = & \,f(y,t) -f(0,t)= \int_0^1 \frac{d}{ds} f(sy,t)\,\mathrm{d}s = \int_0^1 y\cdot \nabla f(sy,t)\mathrm{d}s \\
  = & \int_0^1 y\cdot g(sy) \,\mathrm{d}s =: q(y),
\end{split}
\end{equation}
that is,
\begin{equation}\label{eq:pres2}
  p(x,t)= \frac{1}{(T-t)^{\frac{2\alpha}{1+\alpha}}} q\bigg(\frac{x-x_0}{(T-t)^{\frac{1}{1+\alpha}}}\bigg) + c(t),\quad \forall \,x\in B_\rho(x_0),
\end{equation}
with $c(t)=p(x_0,t)$. Inserting \eqref{eq:pres2} into \eqref{eq:ueq2} yields the equation \eqref{eq:keyeq} on $\mathbb{R}^N$.
In a similar deduction as in \cite[Lemma 2.1]{BroS}, we can prove that $q(y)$ is a tempered distribution.

Now we show that $q$ and $I$ are equal up to a first-order harmonic polynomial.
Since they both satisfy the Laplace equation $\Delta I =-\mathrm{div}\mathrm{div}(u\otimes u)= \Delta q$,
and are both tempered distributions on $\mathbb{R}^N$, the difference
\begin{equation}\label{eq:qi1}
  q - I=:h
\end{equation}
is a harmonic polynomial. In the following we prove the order of $h$ is at most one.
For all $|y|\leq \frac{\rho}{2(T-t)^{\frac{1}{1+\alpha}}}$, from \eqref{eq:pexp} and \eqref{eq:vpSelf} we have
\begin{equation}\label{eq:qi2}
\begin{split}
  &(T-t)^{\frac{2\alpha}{1+\alpha}} p(y (T-t)^{\frac{1}{1+\alpha}},t)  = \\
   =\,& -\frac{1}{N} |u(y)|^2
  + p.v. \int_{|z|\leq \rho} K_{ij}\big(y(T-t)^{\frac{1}{1+\alpha}}-z\big)\;(u_i u_j)
  \Big(\frac{z}{(T-t)^{\frac{1}{1+\alpha}}}\Big)\mathrm{d}z \\
  & \,+ (T-t)^{\frac{2\alpha}{1+\alpha}} \int_{|z|\geq \rho} K_{ij}\big(y(T-t)^{\frac{1}{1+\alpha}}-z\big)
  (v_i v_j)(z,t)\,\mathrm{d}z \\
  = \, & -\frac{1}{N} |u(y)|^2
  + p.v. \int_{|z|\leq \rho (T-t)^{-\frac{1}{1+\alpha}}} K_{ij}(y-z)(u_i u_j)\big(z\big)\mathrm{d}z + \tilde{p}(y,t), \\
  = \, &\,  I(y) -\tilde{I}(y,t) + \tilde{p}(y,t),
\end{split}
\end{equation}
with
\begin{equation*}
  \tilde{p}(y,t):= (T-t)^{\frac{2\alpha}{1+\alpha}} \int_{|z|\geq \rho} K_{ij}\big(y(T-t)^{\frac{1}{1+\alpha}}-z\big)
  (v_i v_j)(z,t)\,\mathrm{d}z,
\end{equation*}
and
\begin{equation*}
  \tilde{I}(y,t) := \int_{|z|\geq \rho(T-t)^{-\frac{1}{1+\alpha}}} K_{ij}(y-z) u_i(z) u_j(z)\,\mathrm{d}z + \bar q\big(y\big).
\end{equation*}
On the other hand, thanks to \eqref{eq:pres2}, we see that 
\begin{equation}\label{eq:qi3}
  (T-t)^{\frac{2\alpha}{1+\alpha}} p(y (T-t)^{\frac{1}{1+\alpha}},t)= q(y) + d(t)
\end{equation}
with $d(t):=(T-t)^{\frac{2\alpha}{1+\alpha}} c(t)$.
Hence, from \eqref{eq:qi1}-\eqref{eq:qi3}, we see that
\begin{equation*}
  |h(y)-d(t)| \leq |\tilde{p}(y,t)| + |\tilde{I}(y,t)|, \qquad \forall |y|\leq \frac{\rho}{2(T-t)^{\frac{1}{1+\alpha}}}.
\end{equation*}
For $\tilde{p}$, from the separation of $y(T-t)^{\frac{1}{1+\alpha}}$ and $z$, we obtain
\begin{equation*}
  |\tilde{p}(y,t)|\lesssim (T-t)^{\frac{2\alpha}{1+\alpha}} \|v\|_{L^2}^2\lesssim (T-t)^{\frac{2\alpha}{1+\alpha}}.
\end{equation*}
For $\tilde{I}$, thanks to \eqref{eq:I2yL}, \eqref{eq:I2key} and \eqref{eq:I2key2}, we get
\begin{equation*}
  |\tilde{I}(y,t)| \lesssim
  \begin{cases}
    (T-t)^{-\frac{2\delta}{1+\alpha}}, \quad &\textrm{if}\;\; 1\lesssim |u(z)|\lesssim |z|^\delta,\delta\in [0,1[, \\
    (T-t)^{\frac{2N}{p(1+\alpha)}}, \quad & \textrm{if}\;\; u\in L^p(\mathbb{R}^N),\,p\in[2,\infty[.
  \end{cases}
\end{equation*}
Since $\alpha>-1$, $\delta\in[0,1[$ and the above estimates hold for all $y\leq \frac{\rho}{2(T-t)^{1/(1+\alpha)}}$, we infer that the order of harmonic polynomial $h(y)$ is at most one and
\begin{equation}\label{eq:dt}
  |d(t)|\lesssim
  \begin{cases}
    1+(T-t)^{\frac{2\alpha}{1+\alpha}}+(T-t)^{-\frac{2\delta}{1+\alpha}}, \quad &\textrm{if}\;\; 1\lesssim |u(z)|\lesssim |z|^\delta,\delta\in [0,1[, \\
    1+(T-t)^{\frac{2\alpha}{1+\alpha}}+(T-t)^{\frac{2N}{p(1+\alpha)}}, \quad & \textrm{if}\;\; u\in L^p(\mathbb{R}^N),\,p\in[2,\infty[,
  \end{cases}
\end{equation}
which proves \eqref{eq:keyeq2}.
In particular, if $\alpha>-1/2$, $\delta\in[0,1/2[$ or $\alpha>-\frac{1}{2}$, $u\in L^p(\mathbb{R}^N)$ ($ p\in[2,\infty[$), moreover $h(y)$ is a uniform constant.
\end{proof}

\begin{lemma}\label{lem:pres}
  Assume that $u\in C^1_{\mathrm{loc}}(\mathbb{R}^N;\mathbb{R}^N)$ is a locally regular vector field.
Suppose $u$ additionally satisfies that
\begin{equation}
\begin{split}
  |u(y)|\lesssim |y|^\delta,\;\;\forall |y|\geq M, \quad&\textrm{with}\quad 0\leq \delta<1\quad\textrm{and}\\
  \int_{|y|\leq L} |u(y)|^2\,\mathrm{d}y \lesssim L^b,\;\; \forall L\geq M,\quad &\textrm{with}\quad 0\leq b\leq  N+2\delta,
\end{split}
\end{equation}
and $M>0$ a fixed number.
Let $q$ be a scalar field defined from $u$ by that for every $|y|\leq L$,
\begin{equation}
\begin{split}
  q(y)= &\, c_0 |u(y)|^2 + A\cdot y + p.v. \int_{\mathbb{R}^N} K_{ij}(y-z) u_i(z) u_j(z)\,\mathrm{d}z\, + \\
  & \,+
  \begin{cases}
  -\int_{|z|\geq M} K_{ij}(z) u_i(z) u_j(z)\,\mathrm{d}z, \quad& \textrm{if  }\delta\in [0,1/2[, \\
  -\int_{|z|\geq M} \big(K_{ij}(z)+y\cdot\nabla K_{ij}(z)\big) u_i(z) u_j(z)\,\mathrm{d}z, \quad & \textrm{if  }\delta\in[1/2,1[,
  \end{cases}
\end{split}
\end{equation}
with $c_0\in\mathbb{R}$, $A\in \mathbb{R}^N$ and $K_{ij}(z)$ ($i,j=1,\cdots,N$) some Calder\'on-Zygmund kernel, then we have
\begin{equation}\label{eq:estq2}
  \int_{|y|\leq L} |q(y)| |u(y)| \,\mathrm{d}y \lesssim
  \begin{cases}
    L^{b+\delta} +  L^{\frac{N+b}{2}+1},\quad & \textrm{if}\;\; (b,\delta)\neq (N+1,\frac{1}{2}),\\
    L^{\frac{N+3}{2}}[\log_2 L], \quad &\textrm{if}\;\; (b,\delta)=(N+1,\frac{1}{2}).
  \end{cases}
\end{equation}
In particular, if $\delta\in [0,\frac{1}{2}[$ and $A=0$, we also have
\begin{equation}\label{eq:estq3}
  \int_{|y|\leq L} |q(y)| |u(y)| \,\mathrm{d}y \lesssim
  \begin{cases}
    L^{b+\delta},\quad & \textrm{if}\;\; b\geq N-2\delta,(b,\delta)\neq (N,0), \\
    L^N [\log_2 L],\quad & \textrm{if}\;\; (b,\delta)=(N,0), \\
    L^{\frac{N+b}{2}} ,\quad & \textrm{if}\;\; b\leq N-2\delta,(b,\delta)\neq (N,0). \\
  \end{cases}
\end{equation}
\end{lemma}

\begin{proof}[Proof of Lemma \ref{lem:pres}]

We decompose $q(y)$ as
\begin{equation}\label{eq:qy-dec}
  q(y)= c_0 |u(y)|^2 + q_1(y, L) + q_2(y, L) + q_3 (y, L) + q_4(y, L)
\end{equation}
with
\begin{equation*}
\begin{split}
  & q_1(y, L)= A\cdot y, \quad\quad q_2(y, L) = p.v. \int_{|y|\leq 2L} K_{ij}(y-z) u_i(z) u_j(z) \,\mathrm{d}z, \\
  & q_3(y, L)=
  \begin{cases}
    \int_{|z|\geq 2 L} \big(K_{ij}(y-z)-K_{ij}(z)\big)u_i(z) u_j(z)\,\mathrm{d}z, \quad &\textrm{if}\;\; \delta\in [0,\frac{1}{2}[\\
    \int_{|z|\geq 2 L} \big(K_{ij}(y-z)-K_{ij}(z)-y\cdot\nabla K_{ij}(z)\big)u_i(z) u_j(z)\,\mathrm{d}z, \quad &\textrm{if}\;\; \delta\in [\frac{1}{2},1[,
  \end{cases}
  \\
  & q_4(y, L)=
  \begin{cases}
    -\int_{M\leq |z|\leq 2L} K_{ij}(z) u_i(z) u_j(z)\,\mathrm{d}z, \quad &\textrm{if}\;\; \delta\in[0,\frac{1}{2}[, \\
    -\int_{M\leq |z|\leq 2L} \big(K_{ij}(z) + y\cdot\nabla K_{ij}(z)\big)u_i(z) u_j(z)\,\mathrm{d}z, \quad &\textrm{if}\;\; \delta\in[\frac{1}{2},1[.
  \end{cases}
\end{split}
\end{equation*}
We first directly have $\int_{|y|\leq L}|u(y)|^3\,\mathrm{d}y\lesssim L^{b+\delta}$, and
\begin{equation*}
  \int_{|y|\leq L} |q_1(y,L)| |u(y)|\,\mathrm{d}y \leq |A| L^{N/2+1} \Big(\int_{|y|\leq L} |u(y)|^2\,\mathrm{d}y\Big)^{1/2} \lesssim |A| L^{\frac{N+b}{2}+ 1}.
\end{equation*}
For the term involving $q_2(y,L)$,
by the H\"older inequality and Calder\'on-Zygmund theorem, we get
\begin{equation*}
\begin{split}
  \int_{|y|\leq L}|q_2(y,L)| |u(y)|\,\mathrm{d}y & \leq \Big(\int_{|y|\leq L}|q_2(y,L)|^{\frac{3}{2}}\,\mathrm{d}y\Big)^{\frac{2}{3}}
  \Big( \int_{|y|\leq L} |u(y)|^3\,\mathrm{d}y\Big)^{\frac{1}{3}} \\
  & \lesssim \int_{|y|\leq 2L} |u(y)|^3\,\mathrm{d}y \lesssim L^{b+\delta}.
\end{split}
\end{equation*}
For the term containing $q_3(y,L)$, using the support property and the dyadic decomposition again, we infer that if $\delta\in [0,1/2[$,
\begin{equation*}
\begin{split}
  \int_{|y|\leq L} |q_3(y,L)| |u(y)|\,\mathrm{d}y & \lesssim L^{N+\delta} \sup_{|y|\leq L} |q_3(y,L)| \\
  & \lesssim L^{N+\delta} \sup_{|y|\leq L} \bigg(\sum_{k=1}^\infty \int_{2^k L\leq |z|\leq 2^{k+1}L}
  \frac{|y|}{|z|^{N+1}} |u(z)|^2\,\mathrm{d}z\bigg) \\
  & \lesssim L^{N+\delta+1} \sum_{k=1}^\infty \frac{1}{(2^k L)^{N+1}} \int_{|z|\sim 2^k L} |u(z)|^2\,\mathrm{d}z \\
  & \lesssim L^{N+\delta+1}\sum_{k=1}^\infty (2^k L)^{b-N-1}\lesssim L^{b+\delta},
\end{split}
\end{equation*}
and if $\delta\in [1/2,1[$,
\begin{equation*}
\begin{split}
  \int_{|y|\leq L} |q_3(y,L)| |u(y)|\,\mathrm{d}y & \lesssim L^{N+\delta} \sup_{|y|\leq L} |q_3(y,L)| \\
  & \lesssim L^{N+\delta} \sup_{|y|\leq L} \bigg(\sum_{k=1}^\infty \int_{2^k L\leq |z|\leq 2^{k+1}L}
  \frac{|y|^2}{|z|^{N+2}} |u(z)|^2\,\mathrm{d}z\bigg) \\
  & \lesssim L^{N+\delta+2} \sum_{k=1}^\infty \frac{1}{(2^k L)^{N+2}} \int_{|z|\sim 2^k L} |u(z)|^2\,\mathrm{d}z \\
  & \lesssim L^{N+\delta+2} \sum_{k=1}^\infty (2^k L)^{b-N-2}\lesssim L^{b+\delta}.
\end{split}
\end{equation*}
For the last term, from H\"older's inequality and the dyadic decomposition we deduce that if $\delta\in [0,\frac{1}{2}[$,
\begin{equation*}
\begin{split}
  \int_{|y|\leq L} |q_4(y,L)| |u(y)|\,\mathrm{d}y & \lesssim L^{N/2} \Big(\int_{|y|\leq L}|u(y)|^2\,\mathrm{d}y\Big)^{1/2} \Big(\sup_{|y|\leq L}|q_4(y,L)|\Big) \\
  & \lesssim L^{\frac{N+b}{2}} \sum_{k=-1}^{[\log_2 \frac{L}{M}]}\int_{\frac{L}{2^{k+1}}\leq |z|\leq \frac{L}{2^k}} \frac{1}{|z|^N} |u(z)|^2\,\mathrm{d}z \\
  & \lesssim L^{\frac{N+b}{2}} \sum_{k=-1}^{[\log_2 \frac{L}{M}]} \Big(\frac{L}{2^k}\Big)^{-N+b} \lesssim
  \begin{cases}
    L^{\frac{3b-N}{2}},\quad &\textrm{if}\;\; b>N, \\
    L^N [\log_2 L],\quad &\textrm{if}\;\; b=N, \\
    L^{\frac{N+b}{2}},\quad &\textrm{if}\;\; b<N,
  \end{cases} \\
  & \lesssim
  \begin{cases}
    L^{b+\delta},\quad & \textrm{if}\;\; b\geq N,(b,\delta)\neq (N,0), \\
    L^N [\log_2 L],\quad & \textrm{if}\;\; (b,\delta)=(N,0), \\
    L^{\frac{N+b}{2}} ,\quad & \textrm{if}\;\; b<N, \\
  \end{cases}
\end{split}
\end{equation*}
and if $\delta\in[\frac{1}{2},1[$,
\begin{equation*}
\begin{split}
  \int_{|y|\leq L} |q_4(y,L)| |u(y)|\,\mathrm{d}y & \lesssim L^{N/2} \Big(\int_{|y|\leq L}|u(y)|^2\,\mathrm{d}y\Big)^{1/2} \Big(\sup_{|y|\leq L}|q_4(y,L)|\Big) \\
  & \lesssim L^{\frac{N+b}{2}} \sum_{k=-1}^{[\log_2 \frac{L}{M}]}\int_{\frac{L}{2^{k+1}}\leq |z|\leq \frac{L}{2^k}} \Big(\frac{1}{|z|^N} + \frac{L}{|z|^{N+1}}\Big) |u(z)|^2\,\mathrm{d}z \\
  & \lesssim L^{\frac{N+b}{2}+1} \sum_{k=-1}^{[\log_2 \frac{L}{M}]} \Big(\frac{L}{2^k}\Big)^{-N-1+b} \lesssim
  \begin{cases}
    L^{\frac{3b-N}{2}},\quad &\textrm{if}\;\; b>N+1, \\
    L^{N+\frac{3}{2}} [\log_2 L],\quad &\textrm{if}\;\; b=N+1, \\
    L^{\frac{N+b}{2}+1},\quad &\textrm{if}\;\; b<N+1,
  \end{cases}
  \\ & \lesssim
  \begin{cases}
    L^{b+\delta},\quad & \textrm{if}\;\; b\geq N+1,\,(b,\delta)\neq (N+1,\frac{1}{2}), \\
    L^{N+\frac{3}{2}}[\log_2 L], \quad &\textrm{if}\;\; (b,\delta)= (N+1,\frac{1}{2}), \\
    L^{\frac{N+b}{2}+1},\quad & \textrm{if}\;\; b<N+1.
  \end{cases}
\end{split}
\end{equation*}
Collecting the above estimates leads to \eqref{eq:estq2}.
\end{proof}

\section{Proof of Theorem \ref{thm:SS}}\label{sec:thm:SS}


As mentioned in the introduction section, the starting point of the main proof is the following local energy inequality for profiles $(u,q)$:
\begin{equation}\label{eq:locEne2}
\begin{split}
  \bigg|\frac{1}{l_2^{N-2\alpha}}
  \int_{|y|\leq l_2} |u(y)|^2 \phi\Big(\frac{y }{l_2}\Big) \mathrm{d}y -
  \frac{1}{l_1^{N-2\alpha}} \int_{|y|\leq l_1} |u(y)|^2 \phi\Big(\frac{y}{l_1}\Big)\mathrm{d}y\bigg|
  \leq \,C \int_{\frac{1}{2}l_1\leq |y|\leq l_2} \frac{|u|^3 + |q| |u| }{|y|^{N-2\alpha+1}} \mathrm{d}y,
\end{split}
\end{equation}
where $0<l_1<l_2$, and $\phi\in \mathcal{D}(\mathbb{R}^N)$
is a radial smooth cutoff function such that $0\leq \phi\leq 1$, $\phi\equiv 1$ on $B_{1/2}(0)$ and
$\phi\equiv 0$ on $B_1^c(0)$. This inequality \eqref{eq:locEne2} is derived from
the energy equality of the original velocity,
\begin{equation}\label{eq:locEne0}
\begin{split}
  \int_{\mathbb{R}^N} |v(t_2,x)|^2 \phi(x)\,\mathrm{d}x - \int_{\mathbb{R}^N} |v(t_1,x)|^2 \phi(x)\,\mathrm{d}x
  =  \int_{t_1}^{t_2}\int_{\mathbb{R}^N}
  \big(|v|^2v + 2(p-d(t))v\big)\cdot\nabla \phi(x) \,\mathrm{d}x \mathrm{d}t,
\end{split}
\end{equation}
where $0<t_1<t_2<T$. By considering \eqref{eq:locEne0} on the region of self-similarity,
and inserting the self-similar scenario \eqref{eq:vpSelf}-\eqref{eq:vpSelf2} into \eqref{eq:locEne0}, and through changing of variables,
we can show \eqref{eq:locEne2} (e.g. see \cite{ChaS}). Notice that $l_i=(T-t_i)^{-\frac{1}{1+\alpha}}$, $i=1,2$ in \eqref{eq:locEne2}.

\subsection{Proof of Theorem \ref{thm:SS}-(1)}

First we consider the case $-1<\alpha<-\delta$. Note that from \eqref{eq:asum}, we have that in this case
\begin{equation*}
  \frac{1}{{\l_2}^{N-2\alpha}}\int_{|y|\leq l_2} |u(y)|^2 \mathrm{d}y \lesssim {l_2}^{2\delta+2\alpha}\longrightarrow 0,
  \quad \textrm{as}\quad l_2\rightarrow \infty.
\end{equation*}
Thus by letting $l_1=2L\gg1$ and $l_2\rightarrow \infty$, we get
\begin{equation}\label{eq:uL2key1}
\begin{split}
  \int_{|y|\leq L} |u(y)|^2 \mathrm{d}y \leq
  C L^{N-2\alpha} \int_{|y|\geq L} \frac{|u|^3 + |q| |u|}{|y|^{N-2\alpha+1}} \mathrm{d}y,
\end{split}
\end{equation}
where $q$ takes the formula as \eqref{eq:qy}.
By the dyadic decomposition, we infer that
\begin{equation}\label{eq:G2Lest}
\begin{split}
  & \, C L^{N-2\alpha} \sum_{k=0}^\infty \int_{2^k L\leq |y|\leq 2^{k+1}L}
  \frac{|u|^3 + |q| |u|}{|y|^{N-2\alpha+1}} \,\mathrm{d}y  \\
  \leq\; & \frac{C}{L} \sum_{k=0}^\infty  2^{-k(N-2\alpha+1)} \int_{2^k L\leq |y|\leq 2^{k+1}L} \big(|u(y)|^3+ |q(y)| |u(y)|\big)\,\mathrm{d}y
\end{split}
\end{equation}
By using the following estimate
\begin{equation}\label{eq:uL2est}
  \int_{|y|\leq L} |u(y)|^2\,\mathrm{d}y\lesssim L^{N+2\delta},\quad \forall L\gg1,
\end{equation}
\eqref{eq:estq2} in Lemma \ref{lem:pres} ensures that
\begin{equation*}
\begin{split}
  \int_{|y|\leq 2^{k+1}L} |u(y)| |q(y)|\,\mathrm{d}y & \lesssim
  \begin{cases}
    (2^k L)^{N+3\delta} + (2^k L)^{N+\delta +1},\quad &\textrm{if}\;\; \delta\neq\frac{1}{2}, \\
    (2^k L)^{N+\frac{3}{2}}[\log_2 (2^k L)],\quad &\textrm{if}\;\; \delta=\frac{1}{2},
  \end{cases} \\
  & \lesssim
  \begin{cases}
  (2^k L)^{N+3\delta},\quad &\textrm{if   }\delta> \frac{1}{2}, \\
  (2^k L)^{N+\frac{3}{2}+\epsilon}, \quad & \textrm{if   }\delta=\frac{1}{2},\\
  (2^k L)^{N+\delta+1} ,\quad &\textrm{if    } \delta<\frac{1}{2},
  \end{cases}
\end{split}
\end{equation*}
with $0<\epsilon\ll 1/2$ a small number.
Thus for all $-1<\alpha<-\delta$, we first obtain a rough bound
\begin{equation}\label{eq:uL2}
\begin{split}
  \int_{|y|\leq L} |u(y)|^2 \mathrm{d}y & \leq
  \begin{cases}
    \frac{C}{L} \sum_{k=0}^\infty 2^{-k(N-2\alpha+1)} (2^k L)^{\max\{N+3\delta,N+1+\delta \}},\quad& \textrm{if}\;\; \delta\neq \frac{1}{2}, \\
    \frac{C}{L} \sum_{k=0}^\infty 2^{-k(N-2\alpha+1)} (2^k L)^{N+\frac{3}{2}+\epsilon},\quad& \textrm{if}\;\; \delta= \frac{1}{2}, \\
  \end{cases} \\
  & \leq
  \begin{cases}
  C L^{N+3\delta-1},\quad &\textrm{if  }\delta\in ]\frac{1}{2},1[, \\
  C L^{N+\delta+\epsilon} ,\quad &\textrm{if   } \delta\in ]0,\frac{1}{2}].
  \end{cases}
\end{split}
\end{equation}
Next we will use \eqref{eq:uL2} to show an more refined bound. By using Lemma \ref{lem:pres} again, and noting that
\begin{equation}\label{eq:fact1}
  \max\{b+\delta, (N+b)/2+1\}=
  \begin{cases}
    b+\delta,\quad & \textrm{if    }b\geq N+2(1-\delta), \\
    \frac{N+b}{2}+1,\quad & \textrm{if   }b<N+2(1-\delta),
  \end{cases}
\end{equation}
we get
\begin{equation}\label{eq:est4}
  \int_{|y|\leq 2^{k+1}L} |u(y)| |q(y)|\,\mathrm{d}y \lesssim
  \begin{cases}
  (2^k L)^{N+4\delta-1},\quad &\textrm{if  }\delta\in [\frac{3}{5},1[, \\
  (2^k L)^{N+\frac{3\delta+1}{2}},\quad &\textrm{if  }\delta\in ] \frac{1}{2}, \frac{3}{5}], \\
  (2^k L)^{N+\frac{\delta+\epsilon}{2}+1},\quad &\textrm{if  }\delta\in ]0, \frac{1}{2}],
  \end{cases}
\end{equation}
Plugging it into \eqref{eq:G2Lest}, we have
\begin{equation}\label{eq:G2}
\begin{split}
  \int_{|y|\leq L} |u(y)|^2 \,\mathrm{d}y  & \leq
  \begin{cases}
  \frac{C}{L}\sum_{k=0}^\infty 2^{-k(N-2\alpha+1)}  (2^k L)^{N+4\delta-1}, \quad &\textrm{if  }\delta\in [\frac{3}{5},1[,\\
  \frac{C}{L}\sum_{k=0}^\infty 2^{-k(N-2\alpha+1)}  (2^k L)^{N+\frac{3\delta+1}{2}},\quad &\textrm{if  }\delta\in ] \frac{1}{2}, \frac{3}{5}], \\
  \frac{C}{L}\sum_{k=0}^\infty 2^{-k(N-2\alpha+1)}  (2^k L)^{N+\frac{\delta+\epsilon}{2}+1},\quad &\textrm{if  }\delta\in ]0, \frac{1}{2}],
  \end{cases} \\
  & \leq
  \begin{cases}
   C L^{N+4\delta-2},\quad &\textrm{if  }\delta\in [\frac{3}{5},1[, \\
   C L^{N+\frac{3\delta-1}{2}},\quad &\textrm{if  }\delta\in ] \frac{1}{2}, \frac{3}{5}], \\
   C L^{N+\frac{\delta+\epsilon}{2}},\quad &\textrm{if  }\delta\in ]0, \frac{1}{2}].
  \end{cases}
\end{split}
\end{equation}
We can repeat the above process for $n+1$ times to show that
\begin{equation}\label{eq:keyeq3}
  \int_{|y|\leq L} |u(y)|^2\,\mathrm{d}y \leq
  \begin{cases}
    C L^{N+2\delta+ (n+1)(\delta-1)},\quad & \textrm{if   }\delta\in [\frac{n+2}{n+4},1[, \\
    C L^{N+\frac{2\delta+ n(\delta-1)}{2}},\quad &\textrm{if   }\delta\in [\frac{n+1}{n+3}, \frac{n+2}{n+4}], \\
    C L^{N+\frac{2\delta + (n-1)(\delta-1)}{2^2}},\quad & \textrm{if   }\delta\in [\frac{n}{n+2},\frac{n+1}{n+3}],\\
    \cdots\quad \cdots \\
    C L^{N+ \frac{2\delta + (\delta-1)}{2^n}},\quad & \textrm{if   }\delta\in ]\frac{1}{2},\frac{3}{5}], \\
    C L^{N + \frac{\delta+\epsilon}{2^n}},\quad & \textrm{if   }\delta\in ]0,\frac{1}{2}].
  \end{cases}
\end{equation}
For each $\delta\in]0,\frac{1}{2}]$, and for $n$ sufficiently large, we get that the power of $L$ is less than $N+\epsilon_0$ for $\epsilon_0>0$
($\epsilon_0$ is the number appearing in \eqref{eq:asum3});
while for each $\delta\in ]\frac{1}{2},1[$, there is some $m\in \mathbb{N}^+$ so that $\delta\in ]\frac{m+1}{m+3},\frac{m+2}{m+4}]$, thus after repeating the above process for
$m+n+1$ times, we get
\begin{equation*}
  \int_{|y|\leq L}|u(y)|^2 \,\mathrm{d}y\leq C L^{N+ \frac{2\delta + m(\delta-1)}{2^{n+1}}},\quad \textrm{for  }\delta\in \big]\frac{m+1}{m+3},\frac{m+2}{m+4}\big],
\end{equation*}
and for $n$ large enough, we infer that the power of $L$ is also less than $N+\epsilon_0$.
But this obviously contradicts with the following estimation from the condition \eqref{eq:asum3}
\begin{equation}\label{eq:fact0}
  \int_{|y|\leq L}|u(y)|^2\,\mathrm{d}y \gtrsim \int_{M\leq|y|\leq L}|y|^{2\epsilon_0}\,\mathrm{d}y \gtrsim L^{N+2\epsilon_0},
\end{equation}
which means there is no possibility to admit nontrivial velocity profiles in the case $-1<\alpha<-\delta$.

Next we consider the case $\alpha\geq -\delta$, and for all $\alpha\geq -\delta$ and $\delta\in ]0,1[$, we prove 
\begin{equation}\label{eq:conc2}
  \int_{|y|\leq L}|u(y)|^2\,\mathrm{d}y \leq C L^{N-2\alpha},\quad \forall L\gg1,
\end{equation}
from which and \eqref{eq:fact0},
we see that the range of $\alpha$ admitting possible nontrivial profiles belongs to $\{\alpha:-\delta\leq \alpha\leq -\epsilon_0\}$.
By letting $l_1=2M$ and $l_2=2L\gg1$, we begin with \eqref{eq:locEne2} to get
\begin{equation}\label{eq:uL2H}
  \int_{|y|\leq L} |u(y)|^2 \mathrm{d}y \leq C L^{N-2\alpha} +  C L^{N-2\alpha}  \int_{M\leq |y|\leq 2L} \frac{|u|^3+ |q| |u|}{|y|^{N-2\alpha+1}} \mathrm{d}y.
\end{equation}
By the dyadic decomposition, we infer that
\begin{equation}\label{eq:H2L}
\begin{split}
   & C L^{N-2\alpha}  \int_{M\leq |y|\leq 2L} \frac{|u|^3+ |q| |u|}{|y|^{N-2\alpha+1}} \mathrm{d}y \\
   \leq\, &\, C L^{N-2\alpha} \sum_{k=-1}^{[\log_2 \frac{L}{M}]} \int_{\frac{L}{2^{k+1}}\leq |y|\leq \frac{L}{2^k }}
  \frac{|u(y)|^3 + |q(y)| |u(y)|}{|y|^{N-2\alpha +1}} \mathrm{d}y \\
  \leq\, & \, \frac{C}{L} \sum_{k=-1}^{[\log_2 \frac{L}{M}]} 2^{k(N-2\alpha+1)}\int_{\frac{L}{2^{k+1}}\leq |y|\leq \frac{L}{2^k}}
  \big(|u(y)|^3 + |q(y)| |u(y)|\big)\,\mathrm{d}y.
\end{split}
\end{equation}
By using the formula of $q$ as \eqref{eq:qy},
and from \eqref{eq:asum2}, \eqref{eq:uL2est} and \eqref{eq:estq2} in Lemma \ref{lem:pres}, we have a rough bound:
\begin{equation*}
\begin{split}
  \int_{|y|\leq L}|u(y)|^2\,\mathrm{d}y & \leq
  \begin{cases}
    C L^{N-2\alpha} +  \frac{C}{L}\sum_{k=-1}^{[\log_2 \frac{L}{M}]}2^{k(N-2\alpha+1)}
    \Big(\Big(\frac{L}{2^k}\Big)^{N+3\delta}+  \Big(\frac{L}{2^k}\Big)^{N+\delta+1}\Big),\quad & \textrm{if}\;\; \delta\neq \frac{1}{2}, \\
    C L^{N-2\alpha} + \frac{C}{L} \sum_{k=-1}^{[\log_2 \frac{L}{M}]} 2^{k(N-2\alpha+1)} \Big(\frac{L}{2^k}\Big)^{N+3/2} [\log \frac{L}{2^k}],\quad &\textrm{if}\;\; \delta=\frac{1}{2},
  \end{cases}
  \\
  & \leq
  \begin{cases}
    C L^{N-2\alpha}[\log_2 L],\quad & \textrm{if}\;\; \alpha \leq -\frac{3\delta -1}{2},\,\delta\in ]\frac{1}{2},1[, \\
    C L^{N+3\delta-1},\quad  & \textrm{if}\;\;  \alpha> -\frac{3\delta-1}{2},\, \delta\in ]\frac{1}{2},1[, \\
    C L^{N-2\alpha} [\log_2 L]^2,\quad & \textrm{if}\;\; \alpha\leq -\frac{\delta}{2},\, \delta\in]0,\frac{1}{2}], \\
    C L^{N+\delta+\epsilon},\quad & \textrm{if}\;\; \alpha > -\frac{\delta}{2},\, \delta\in ]0,\frac{1}{2}],
  \end{cases}
\end{split}
\end{equation*}
with $0<\epsilon\ll \delta $ a small number.
If $\alpha\leq -\frac{3\delta-1}{2}$, $\delta\in ]\frac{1}{2},1[$ or $\alpha\leq -\frac{\delta}{2}$, $\delta\in ]0,\frac{1}{2}]$,
we can improve the bound to drop the additional logarithmic term: indeed, let $0<\epsilon<1$ be a small number chosen later, then we get
\begin{equation*}
  \int_{|y|\leq L}|u(y)|^2\,\mathrm{d}y \leq C_\epsilon L^{N-2\alpha+\epsilon},\quad \textrm{if}\;\;
  \begin{cases}
  \alpha\leq -\frac{3\delta-1}{2}, \delta\in ]\frac{1}{2},1[,\;\textrm{ or } \; \\ \alpha\leq -\frac{\delta}{2}, \delta\in ]0,\frac{1}{2}],
  \end{cases}
\end{equation*}
and inserting this estimate into \eqref{eq:H2L} yields that for all such $(\alpha,\delta)$,
\begin{equation}\label{eq:keyeq5}
\begin{split}
  \int_{|y|\leq L} |u(y)|^2\,\mathrm{d}y \lesssim &
  L^{N-2\alpha} + \frac{1}{L} \sum_{k=-1}^{[\log_2 \frac{L}{M}]} 2^{k(N-2\alpha+1)} \Big(\Big(\frac{L}{2^k}\Big)^{N-2\alpha+\epsilon +\delta} +
  \Big(\frac{L}{2^k}\Big)^{N-\alpha+\epsilon/2 +1}\Big) \\
  \lesssim &
    L^{N-2\alpha} + L^{N-2\alpha+\epsilon+\delta-1}\sum_{k=-1}^{\log_2 (\frac{L}{M})} 2^{k(1-\epsilon-\delta)} + L^{N-\alpha+\epsilon/2} \sum_{k=-1}^{[\log_2 \frac{L}{M}]} 2^{k(-\alpha-\epsilon/2)} \\
  \lesssim & L^{N-2\alpha},
\end{split}
\end{equation}
as long as $0<\epsilon <\min\{1-\delta,-2\alpha\}$, which can be satisfied by choosing 
\begin{equation*}
  \epsilon = 
  \begin{cases}
    \frac{1-\delta}{2},\quad & \textrm{for}\;\; \delta\in ]\frac{1}{2},1[, \\
    \frac{\delta}{2},\quad & \textrm{for}\;\;\delta\in ]0,\frac{1}{2}].
  \end{cases}
\end{equation*}
If $\alpha>-\frac{3\delta-1}{2}$ for $\delta\in ]\frac{1}{2},1[$ or $\alpha>-\frac{\delta}{2}$ for $\delta\in ]0,\frac{1}{2}]$,
we will use the iterative method to reduce the power of $L$. Thanks to \eqref{eq:estq2} in Lemma \ref{lem:pres}, we get \eqref{eq:est4} with $2^{k+1}L$ replaced by $\frac{L}{2^{k}}$,
and by inserting it into \eqref{eq:H2L}, we find
\begin{equation}\label{eq:est1}
\begin{split}
  \int_{|y|\leq L}|u(y)|^2\,\mathrm{d}y \lesssim &
  \begin{cases}
    L^{N-2\alpha} + \frac{1}{L} \sum_{k=-1}^{[\log_2 \frac{L}{M}]} 2^{k(N-2\alpha+1)} \Big(\frac{L}{2^k}\Big)^{N+4\delta-1}, \quad & \textrm{if}\;\; \alpha>-\frac{3\delta-1}{2},\delta\in [\frac{3}{5},1[, \\
    L^{N-2\alpha} + \frac{1}{L} \sum_{k=-1}^{[\log_2 \frac{L}{M}]} 2^{k(N-2\alpha+1)} \Big(\frac{L}{2^k}\Big)^{N+(3\delta-1)/2+1}, \quad & \textrm{if}\;\;  \alpha>-\frac{3\delta-1}{2}, \delta\in ]\frac{1}{2},\frac{3}{5}],\\
    L^{N-2\alpha} + \frac{1}{L} \sum_{k=-1}^{[\log_2 \frac{L}{M}]} 2^{k(N-2\alpha+1)} \Big(\frac{L}{2^k}\Big)^{N+(\delta+\epsilon)/2+1}, \quad & \textrm{if}\;\; \alpha>-\frac{\delta}{2}, \delta\in ]0,\frac{1}{2}],
  \end{cases} \\
  \lesssim &
  \begin{cases}
    L^{N-2\alpha}[\log_2 L],\quad & \textrm{if}\;\; \alpha\in ]-\frac{3\delta-1}{2},-(2\delta-1)],\delta\in [\frac{3}{5},1[, \\
    L^{N+4\delta-2},\quad & \textrm{if}\;\; \alpha > -(2\delta-1), \delta \in [\frac{3}{5},1[, \\
    L^{N-2\alpha}[\log_2 L],\quad & \textrm{if}\;\; \alpha\in ]-\frac{3\delta-1}{2},-\frac{3\delta-1}{2^2}],\delta\in ]\frac{1}{2},\frac{3}{5}], \\
    L^{N+\frac{3\delta-1}{2}},\quad & \textrm{if}\;\; \alpha > -\frac{3\delta-1}{2^2}, \delta \in ]\frac{1}{2},\frac{3}{5}], \\
    L^{N-2\alpha}[\log_2 L],\quad & \textrm{if}\;\; \alpha\in ]-\frac{\delta}{2},-\frac{\delta+\epsilon}{2^2}],\delta\in ]0,\frac{1}{2}], \\
    L^{N+\frac{\delta+\epsilon}{2}},\quad & \textrm{if}\;\; \alpha > -\frac{\delta+\epsilon}{2^2}, \delta \in ]0,\frac{1}{2}].
  \end{cases}
\end{split}
\end{equation}
If $\alpha\in ]-\frac{3\delta-1}{2}$, $-(2\delta-1)],\delta\in [\frac{3}{5},1[$, or $\alpha\in ]-\frac{3\delta-1}{2}$, $-\frac{3\delta-1}{2^2}],\delta\in ]\frac{1}{2},\frac{3}{5}],$
or $\alpha\in ]-\frac{\delta}{2},-\frac{\delta+\epsilon}{2^2}]$, $\delta\in ]0,\frac{1}{2}]$, we can also improve the above bound by removing the logarithmic term:
indeed, this is quite similar to the obtaining of \eqref{eq:keyeq5}, and we only need to choose $0<\epsilon< 1$ so that $\epsilon<\min\{1-\delta, -2\alpha\}$, e.g., set
\begin{equation*}
\epsilon=
\begin{cases}
  \frac{1-\delta}{2},\quad &\textrm{for}\;\; \delta\in [\frac{3}{5},1[, \\
  \frac{\delta}{2^2}, \quad & \textrm{for} \;\; \delta\in]0,\frac{3}{5}[,
\end{cases}
\end{equation*}
then the bound of $\int_{|y|\leq L}|u(y)|^2\,\mathrm{d}y$ can be similarly improved from $C_\epsilon L^{N-2\alpha+\epsilon}$ to the expected $C_\epsilon L^{N-2\alpha}$.
If $\alpha > -(2\delta-1), \delta \in [\frac{3}{5},1[$, or $\alpha > -\frac{3\delta-1}{2^2}, \delta \in ]\frac{1}{2},\frac{3}{5}]$, or
$\alpha > -\frac{\delta+\epsilon}{2^2}, \delta \in ]0,\frac{1}{2}]$, we can further improve the bound as obtaining \eqref{eq:est1}. In conclusion, after repeating the above process for $n+1$ times, we infer
\begin{equation}\label{eq:keyeq4}
  \int_{|y|\leq L} |u(y)|^2\,\mathrm{d}y \leq
  \begin{cases}
    C L^{N+2\delta+ (n+1)(\delta-1)},\quad & \textrm{if   }\alpha> -\frac{(n+3)\delta-(n+1)}{2},\delta\in [\frac{n+2}{n+4},1[, \\
    C L^{N+\frac{2\delta+ n(\delta-1)}{2}},\quad &\textrm{if   }\alpha> -\frac{(n+2)\delta-n}{2^2}, \delta\in [\frac{n+1}{n+3}, \frac{n+2}{n+4}], \\
    C L^{N+\frac{2\delta + (n-1)(\delta-1)}{2^2}},\quad & \textrm{if   }\alpha>-\frac{(n+1)\delta- (n-1)}{2^3},\delta\in [\frac{n}{n+2},\frac{n+1}{n+3}],\\
    \cdots\quad \cdots \\
    C L^{N+ \frac{2\delta + (\delta-1)}{2^n}},\quad & \textrm{if   }\alpha>-\frac{3\delta-1}{2^{n+1}},\delta\in ]\frac{1}{2},\frac{3}{5}], \\
    C L^{N + \frac{\delta+\epsilon}{2^n}},\quad & \textrm{if   }\alpha>-\frac{\delta+\epsilon}{2^{n+1}},\delta\in ]0,\frac{1}{2}], \\
    C L^{N-2\alpha}, \quad & \textrm{if for other scopes of  } (\alpha,\delta).
  \end{cases}
\end{equation}
From \eqref{eq:keyeq4}, in a similar way as the deduction after \eqref{eq:keyeq3}, we deduce that if $(\alpha,\delta)$ is not in the scope so that $\int_{|y|\leq L}|u(y)|^2\,\mathrm{d}y$
is bounded by $C L^{N-2\alpha}$, then the quantity is instead bounded by $C L^{N+\epsilon_0}$, which leads to a contradiction with \eqref{eq:fact0}
and means that such scopes are incompatible. Hence for all $\alpha\geq -\delta$ and $\delta\in ]0,1[$, we obtain \eqref{eq:conc2}

At last we prove \eqref{eq:conc}, and in order to do that, it suffices to prove the following inequality for all $-\delta\leq \alpha\leq -\epsilon_0$,
\begin{equation}\label{eq:Target}
   \int_{|y|\leq L}|u(y)|^2 \mathrm{d}y \gtrsim L^{N-2\alpha},\quad \forall L\gg1.
\end{equation}
Suppose it is not true, then there exists a sequence of numbers $L_k \gg1$ such that
\begin{equation*}
  \frac{1}{L_k^{N-2\alpha}} \int_{|y|\leq L_k} |u(y)|^2 \, \mathrm{d}y \rightarrow 0,\quad \textrm{as}\;\; L_k\rightarrow\infty.
\end{equation*}
Thus by setting $l_2=L_k\rightarrow\infty$ and $l_1=2 L>0$ in \eqref{eq:locEne2}, we get
\begin{equation}\label{eq:estL2}
  \int_{|y|\leq L} |u(y)|^2 \mathrm{d}y \leq   C L^{N-2\alpha} \int_{|y|\geq L} \frac{|u|^3 + |q| |u|}{|y|^{N-2\alpha+1}} \mathrm{d}y.
\end{equation}
From \eqref{eq:conc2}, and by using the decomposition \eqref{eq:G2Lest} and \eqref{eq:estq2} in Lemma \ref{lem:pres} with $b=N-2\alpha$, we have
\begin{equation*}
\begin{split}
  \int_{|y|\leq L} |u(y)|^2\,\mathrm{d}y  & \leq
  \begin{cases}
  \frac{C}{L} \sum_{k=0}^\infty \frac{1}{2^{k (N-2\alpha + 1)}} (2^k L)^{\max\{N-2\alpha +\delta,N-\alpha+1\}},\quad& \textrm{if}\;\; \alpha\neq -\frac{1}{2},\delta\neq \frac{1}{2}, \\
  \frac{C}{L} \sum_{k=0}^\infty \frac{1}{2^{k (N-2\alpha + 1)}} (2^k L)^{\frac{3}{2}}[\log_2 (2^k L)],\quad &\textrm{if}\;\; \alpha= -\frac{1}{2},\delta= \frac{1}{2},
  \end{cases} \\
  & \leq
  \begin{cases}
  C L^{N-2\alpha+\delta-1}, \quad & \textrm{if   } \alpha\in [-\delta,\delta-1], \delta\in ]\frac{1}{2},1[, \\
  C L^{N+\frac{1}{2}+\epsilon},\quad & \textrm{if   }  \alpha=-\frac{1}{2}, \delta=\frac{1}{2}, \\
  C L^{N-\alpha},\quad & \textrm{if   }  \alpha\in [ \delta-1,-\frac{\epsilon_0}{2}], \delta\in [0,1-\frac{\epsilon_0}{2}],(\alpha,\delta)\neq (-\frac{1}{2},\frac{1}{2})
  \end{cases}
\end{split}
\end{equation*}
with $0<\epsilon\ll1/2$ a small number.
Using this improved estimate and Lemma \ref{lem:pres} again, similarly as above we find
\begin{equation*}
\begin{split}
  \int_{|y|\leq L} |u(y)|^2\,\mathrm{d}y  & \leq
  \begin{cases}
    \frac{C}{L} \sum_{k=0}^\infty \frac{1}{2^{k (N-2\alpha + 1)}} (2^k L)^{N-2\alpha +2\delta-1},\quad &\textrm{if   }\alpha\in [-\delta,\frac{3}{2}(\delta-1)],\delta\in [\frac{3}{5},1[, \\
    \frac{C}{L} \sum_{k=0}^\infty \frac{1}{2^{k (N-2\alpha + 1)}} (2^k L)^{N-\alpha +\frac{\delta-1}{2} +1},\quad &\textrm{if   }\alpha\in [\frac{3}{2}(\delta-1),\delta-1],\delta\in ]\frac{1}{2},1[, \\
    \frac{C}{L} \sum_{k=0}^\infty \frac{1}{2^{k (N-2\alpha + 1)}} (2^k L)^{N+\frac{-\alpha+\epsilon}{2} +1},\quad &\textrm{if   }\alpha\in [\delta-1,-\frac{\epsilon_0}{2}],\delta\in [0,1-\frac{\epsilon_0}{2}], \\
  \end{cases} \\
  & \leq
  \begin{cases}
    L^{N-2\alpha +2\delta-2},\quad &\textrm{if   }\alpha\in [-\delta,\frac{3}{2}(\delta-1)],\delta\in [\frac{3}{5},1[, \\
    L^{N-\alpha +\frac{\delta-1}{2}},\quad &\textrm{if   }\alpha\in [\frac{3}{2}(\delta-1),\delta-1],\delta\in ]\frac{1}{2},1[, \\
    L^{N+\frac{-\alpha+\epsilon}{2}},\quad &\textrm{if   }\alpha\in [\delta-1,-\frac{\epsilon_0}{2}],\delta\in [0,1-\frac{\epsilon_0}{2}], \\
  \end{cases}
\end{split}
\end{equation*}
By repeating the above process for $n+1$ times leads to
\begin{equation*}
\begin{split}
  \int_{|y|\leq L} |u(y)|^2\,\mathrm{d}y
  \leq
  \begin{cases}
    L^{N-2\alpha +(n+1)(\delta-1)},\quad &\textrm{if   }\alpha\in [-\delta,\frac{n+2}{2}(\delta-1)],\delta\in [\frac{n+2}{n+4},1[, \\
    L^{N-\alpha +\frac{n}{2}(\delta-1)},\quad &\textrm{if   }\alpha\in [\frac{n+2}{2}(\delta-1),\frac{n+1}{2}(\delta-1)],\delta\in [\frac{n+1}{n+3},1[, \\
    \cdots\quad \cdots \\
    L^{N-\frac{\alpha}{2^{n-1}} +\frac{1}{2^n}(\delta-1)},\quad &\textrm{if   }\alpha\in [\frac{3}{2}(\delta-1),(\delta-1)],\delta\in ]\frac{1}{2},1[, \\
    L^{N+\frac{-\alpha+\epsilon}{2^n}},\quad &\textrm{if   }\alpha\in [\delta-1,-\frac{\epsilon_0}{2}],\delta\in [0,1-\frac{\epsilon_0}{2}]. \\
  \end{cases}
\end{split}
\end{equation*}
In a similar way as the deduction after \eqref{eq:keyeq3}, and for $n$ large enough, we see that for all $-\delta\leq \alpha\leq-\epsilon_0$,
\begin{equation*}
  \int_{|y|\leq L} |u(y)|^2 \,\mathrm{d}y \lesssim L^{N+\epsilon_0},\quad \forall L\gg1,
\end{equation*}
which clearly contradicts with the estimation \eqref{eq:fact0} derived from \eqref{eq:asum3}, and thus \eqref{eq:Target} is not compatible and \eqref{eq:conc} is followed.

\subsection{Proof of Theorem \ref{thm:SS}-(2)}
Since $\delta<\frac{1}{2}$ and $\alpha>-\frac{1}{2}$, we have $A=0$ in the formula of $q$ \eqref{eq:qy},
and we can use the better estimate \eqref{eq:estq3} instead of \eqref{eq:estq2} in the proof.
We first consider the case $-1<\alpha<-\delta$. Similarly as above, we also begin with \eqref{eq:uL2key1}, and by virtue of \eqref{eq:uL2est} and \eqref{eq:estq3}, we get
\begin{equation*}
  \int_{|y|\leq 2^{k+1} L} |u(y)| |q(y)|\,\mathrm{d}y\lesssim (2^k L)^{N+3\delta},
\end{equation*}
and
\begin{equation*}
  \int_{|y|\leq L}|u(y)|^2\,\mathrm{d}y\leq \frac{C}{L}\sum_{k=0}^\infty \frac{1}{2^{k(N-2\alpha+1)}} (2^k L)^{N+3\delta}\leq C L^{N+3\delta-1}.
\end{equation*}
We can repeatedly use this process to show that
\begin{equation*}
  \int_{|y|\leq L} |u(y)|^2\,\mathrm{d}y \leq C L^{N+2\delta - (n+1)(1-\delta)},
\end{equation*}
as long as $N+2\delta -n(1-\delta) \geq N-2\delta$, that is, $n\leq \frac{4\delta}{1-\delta}$. Set $n_0 = [\frac{4\delta}{1-\delta}]$, then we obtain
\begin{equation*}
  \int_{|y|\leq L} |u(y)|^2\,\mathrm{d}y \leq C L^{N+2\delta - (n_0+1)(1-\delta)}\leq C L^{N-2\delta},
\end{equation*}
which clearly contradicts with the estimation from the condition \eqref{eq:asum4}
\begin{equation}\label{eq:fact2}
  \int_{|y|\leq L} |u(y)|^2\,\mathrm{d}y \gtrsim \int_{M\leq |y|\leq L} 1\,\mathrm{d}y \gtrsim L^N,
\end{equation}
and means that the case $-1<\alpha<\delta$ is not compatible.

Next we consider the case $\alpha \geq -\delta$ to prove \eqref{eq:conc2}. We similarly begin with \eqref{eq:uL2H}, and from \eqref{eq:H2L} and \eqref{eq:estq3} in Lemma \ref{lem:pres},
we see that
\begin{equation*}
  \int_{|y|\leq \frac{L}{2^k}} |u(y)| |q(y)|\,\mathrm{d}y \lesssim \Big(\frac{L}{2^k}\Big)^{N+3\delta},
\end{equation*}
and
\begin{equation}\label{eq:est2}
\begin{split}
  \int_{|y|\leq L} |u(y)|^2\,\mathrm{d}y & \leq C L^{N-2\alpha} + \frac{C}{L}\sum_{k=-1}^{[\log_2 \frac{L}{M}]} 2^{k(N-2\alpha+1)} \Big(\frac{L}{2^k}\Big)^{N+3\delta} \\
  & \leq
  \begin{cases}
    C L^{N+3\delta-1}, \quad & \textrm{if}\;\; \alpha >-\frac{3\delta-1}{2}, \\
    C L^{N-2\alpha} [\log_2 L], \quad & \textrm{if}\;\; \alpha\leq -\frac{3\delta-1}{2}.
  \end{cases}
\end{split}
\end{equation}
For $\alpha\leq -\frac{3\delta-1}{2}$, we can replace the bound by $C_\epsilon L^{N-2\alpha+\epsilon}$ with $0<\epsilon<1-\delta$, then by repeating the process once more, we obtain
\begin{equation}\label{eq:est3}
  \int_{|y|\leq L} |u(y)|^2\,\mathrm{d}y \lesssim L^{N-2\alpha} + \frac{1}{L} \sum_{k=-1}^{[\log_2 \frac{L}{M}]} 2^{k(N-2\alpha+1)} \Big(\frac{L}{2^k}\Big)^{N-2\alpha+\delta +\epsilon}
  \lesssim L^{N-2\alpha}.
\end{equation}
For $\alpha> -\frac{3\delta-1}{2}$, and if $\delta<\frac{1}{3}$, we see that \eqref{eq:est2} contradicts with \eqref{eq:fact2}, and such a case is not compatible; otherwise,
if $\delta\geq \frac{1}{3}$, we can further improve the bound by iteration:
\begin{equation*}
\begin{split}
  \int_{|y|\leq L} |u(y)|^2 \mathrm{d}y & \lesssim L^{N-2\alpha} + \frac{1}{L}\sum_{k=-1}^{[\log_2 \frac{L}{M}]} 2^{k(N-2\alpha +1)} \Big(\frac{L}{2^k}\Big)^{N+4\delta-1} \\
  & \lesssim
  \begin{cases}
    L^{N+4\delta-2},\quad & \textrm{if}\;\; \alpha> -(2\delta-1), \\
    L^{N-2\alpha} [\log_2 L], \quad & \textrm{if}\;\; \alpha \in ]-\frac{3\delta-1}{2},-(2\delta-1)].
  \end{cases}
\end{split}
\end{equation*}
For $\alpha\in ]-\frac{3\delta-1}{2},-(2\delta-1)]$, we can similarly obtain \eqref{eq:est3} in this case; while for $\alpha> -(2\delta-1)$,
the bound $C L^{N+4\delta-2}$ contradicts with \eqref{eq:fact2} due to $\delta<\frac{1}{2}$, which means such a case is incompatible.
Therefore, for all $\alpha\geq -\delta$ and $\delta<\frac{1}{2}$, we have
\begin{equation}
  \int_{|y|\leq L} |u(y)|^2\,\mathrm{d}y \leq C L^{N-2\alpha}, \quad \forall L\gg1.
\end{equation}
Combined with \eqref{eq:fact2}, we furthermore infer that the scope of $\alpha$ admitting possible nontrivial velocity profiles is $\{\alpha:-\delta\leq \alpha\leq 0\}$.

In the end we prove \eqref{eq:conc} for $-\delta\leq \alpha\leq 0$, and it suffices to prove \eqref{eq:Target} for $\alpha$ in this range.
Similarly as above, we begin with \eqref{eq:estL2} to get
\begin{equation*}
  \int_{|y|\leq L}|u(y)|^2\,\mathrm{d}y \leq \frac{C}{L}\sum_{k=0}^\infty \frac{1}{2^{k(N-2\alpha+1)}} (2^k L)^{N-2\alpha +\delta}\leq C L^{N-2\alpha+\delta-1}.
\end{equation*}
By iteration, we can show that, as long as $N+2\alpha-n(1-\delta)\geq N-2\delta$, we have
\begin{equation*}
  \int_{|y|\leq L} |u(y)|^2\,\mathrm{d}y \leq C L^{N-2\alpha- (n+1)(1-\delta)}.
\end{equation*}
Set $n_0'=[\frac{2\alpha+2\delta}{1-\delta}]$, thus we find
\begin{equation*}
  \int_{|y|\leq L} |u(y)|^2\, \mathrm{d}y \leq C L^{N-2\alpha -(n'_0+1)(1-\delta)} \leq C L^{N-2\delta},
\end{equation*}
which contradicts with the estimation \eqref{eq:fact2}, and thus proves \eqref{eq:Target} and \eqref{eq:conc} for $-\delta\leq\alpha\leq 0$.

\vskip0.2cm

\textbf{Acknowledgments.}
Part of the work was done when the author was a post-doctor of Universit\'e Paris-Est from 2012-09 to 2013-08, and
he would like to express his gratitude for the guidance of Prof. R\'egis Monneau and Prof. Marco Cannone.
The author is partially supported by NSFC grant 11401027.

\end{document}